\documentclass[preprint,1p,times,sort&compress]{elsarticle}
\usepackage{amsmath,amsfonts,mathrsfs}
\usepackage{amssymb}

\usepackage{amsthm}
\newtheorem{theorem}{Theorem}[section]
\newtheorem{lemma}[theorem]{Lemma}
\newtheorem{example}[theorem]{Example}
\newdefinition{rem}{Remark}

\numberwithin{equation}{section}
\def\vector{\mathrm{vec}}

\journal{%Journal of Mathematical Analysis and Applications
         }

\begin{document}
\begin{frontmatter}
\title{Solutions and improved perturbation analysis for the matrix equation $X-A^*X^{-p}A=Q\;(p>0)$}
\tnotetext[label1]{The work was supported in part by Natural Science Foundation of China (11201263) and  Natural Science Foundation of Shandong Province (ZR2012AQ004).}
\author{Jing Li}
%\cortext[cor]{Corresponding author.}
\ead{xlijing@sdu.edu.cn}
%\author[SDU]{Yuhai Zhang}

\address{School of Mathematics and Statistics, Shandong University at Weihai, Weihai 264209, P.R. China}
%\address[SDU]{School of Mathematics, Shandong University, Jinan 250100, P.R. China}

\begin{abstract}
In this paper the nonlinear matrix equation $X-A^*X^{-p}A=Q$ with
$p>0$ is investigated. We consider two cases of this equation: the case $p>1$ and the case $0<p<1.$ In the case $p>1$, a new sufficient condition for the existence of a unique positive definite solution for the matrix equation is obtained. A perturbation estimate for the positive definite solution is derived. Explicit expressions of the condition number for the positive definite solution are given. In the case $0<p<1$, a new sharper perturbation bound for the unique positive definite solution is
evaluated.  A new backward error of an approximate solution to the unique positive definite solution is
obtained. The theoretical
results are illustrated by numerical examples.
\end{abstract}

\begin{keyword}
matrix equation \sep positive definite solution \sep
perturbation bound \sep backward error\sep condition number\\
\vspace{1mm}
\end{keyword}
\end{frontmatter}

%%%%%%%%%%%%%%%%%%%%%%%%%Main Text%%%%%%%%%%%%%%%%%%%%%%%%%%%%%%%%%%%%%%%%%%
\section{Introduction}
In this paper we consider the Hermitian positive definite solution of the nonlinear matrix equation
\begin{equation}\label{eq1}                                   %Eq.(1)
X-A^*X^{-p}A=Q,
\end{equation}
where $A$, $Q$ and $X$ are
$n\times n$ complex matrices, $Q$ is a positive definite matrix and $p>0$. This type of nonlinear matrix
equations arises in the analysis of ladder networks, the dynamic
programming, control theory, stochastic filtering, statistics and
many applications \cite{s1,s2,s3,s4,s5,s6,s7}.

In the last few years, Eq.(\ref{eq1}) was investigated in some special cases.
For the nonlinear matrix equations $X-A^*X^{-1}A=Q$ \cite{s13,s24,s25,s26,s44},
 $X-A^*X^{-2}A=Q$ \cite{s15,s17}, $X-A^*X^{-n}A=Q$ \cite{s42,s43} and $X^{s}-A^*X^{-t}A=Q$ \cite{s18},
there were many contributions
in the literature to the solvability, numerical
solutions and perturbation analysis. In addition,
the similar equations $X+A^*X^{-1}A=Q$ \cite{s8,s9,s10,s11,s12,s44,s25,s24,s23}, $X+A^*X^{-2}A=Q$ \cite{s14,s15,s16}, $X+A^*X^{-n}A=Q$ \cite{s45,s43}, $X^{s}+A^*X^{-t}A=Q$
 \cite{s19,s20,s21,s18,s38}, $X+A^*X^{-q}A=Q$ \cite{s28,s34,s41} and
$X\pm\sum\limits_{i=1}^{m}A_{i}^*X^{-1}A_{i}=Q$ \cite{s30,s29,s31} were studied by many scholars.

In
\cite{s28}, a sufficient condition for the equation
$X-A^*X^{-p}A=Q\;(0<p\leq 1)$ to have a unique positive
definite solution was provided.
When the coefficient matrix $A$ is nonsingular, several sufficient conditions for the equation
$X-A^*X^{-q}A=Q\;(q\geq 1)$ to
have a unique positive definite solution were given in \cite{s39}. When the coefficient matrix $A$ is an arbitrary
complex matrix, necessary conditions and sufficient conditions for the existence of positive definite solutions for the equation $X-A^*X^{-q}A=Q\;(q\geq 1)$ were derived in \cite{s40}. Li and Zhang in \cite{s27} proved that there always exists a unique positive definite solution to the equation $X-A^*X^{-p}A=Q\;(0<p<1)$. They also obtained a perturbation bound and a backward error of an approximate solution for the unique solution of the equation $X-A^*X^{-p}A=Q\;(0<p<1)$.
 \iffalse
 But their approaches will become invalid for the case of $p>1.$ Since the equation $X-A^*X^{-p}A=Q\;(p>1)$ does not always have a unique positive definite solution, there are two difficulties for perturbation analysis to the equation $X-A^*X^{-p}A=Q\;(p>1)$. One difficulty is how to find some reasonable restrictions on the coefficient matrices of perturbed equation ensuring this equation has a unique positive definite solution. Another difficulty is how to find an expression of $\Delta X$ which is easy to handle.
\fi

As a continuation of the previous results, the rest of the paper is organized as follows. Section 2 gives some preliminary lemmas that will be needed to develop this work.
In Section 3, a new sufficient condition for Eq.(\ref{eq1}) with $p>1$
existing a unique positive definite solution is derived. In Section 4, a perturbation bound for the positive definite solution to
 Eq.(\ref{eq1}) with $p>1$ is given. In Section 5, applying the integral representation of matrix function, we also discuss the explicit expressions of condition number for the positive definite solution to Eq.(\ref{eq1}) with $p>1$. Furthermore, in Section 6, a new sharper perturbation bound for the unique positive definite solution to Eq.(\ref{eq1}) with $0<p<1$ is evaluated. In Section 7, a new backward error of an approximate solution to Eq.(\ref{eq1}) with $0<p<1$ is
obtained. Finally, several numerical
  examples are presented in Section 8.

We denote by $\mathcal{C}^{n\times n}$ the set of $n\times n$
complex matrices, by $\mathcal{H}^{n\times n}$ the set of $n\times n$
Hermitian matrices, by $I$ the identity matrix, by $\|\cdot\|$ the spectral
norm, by $\|\cdot\|_{F}$ the Frobenius norm and by $\lambda_{\max}(M)$ and
$\lambda_{\min}(M)$ the maximal and minimal eigenvalues of $M$,
respectively. For $A=(a_{1},\dots, a_{n})=(a_{ij})\in
\mathcal{C}^{n\times n}$ and a matrix $B$, $A\otimes B=(a_{ij}B)$
is a Kronecker product, and $\vector A$
is a vector defined by $\vector A=(a_{1}^{T},\dots,
a_{n}^{T})^{T}$. For $X, Y\in\mathcal{H}^{n\times n}$, we write $X\geq Y$(resp. $X>Y)$ if
$X-Y$ is Hermitian positive semi-definite (resp. definite). Let
$\overline{\kappa}=\lambda_{\max}(A^*A)$, $\underline{\kappa}=\lambda_{\min}(A^*A)$.
%%%%%%%%%%%%%%%%%%%%%%%%section 2 Existence of the unique  solution%%%%%%%%%%%%%%%%%%%%%%%%%%%%%%%
\section{Priliminaries}
In this section, we will give some preliminary lemmas that will be needed to develop this work.

\begin{lemma}  \label{lem1}                                                             %Lemma 1
\cite{s27}
For every positive definite matrix $X\in\mathcal{H}^{n\times n}$, if $0<p<1,$ then
\begin{enumerate}
\item[$(i)$] $X^{-p}=\displaystyle\frac{\sin
p\,\pi}{\pi}\!\!\int^\infty_0\!\!\!(\lambda\,I+X)^{-1}\lambda^{-p}\,\emph{d}\lambda.$
\item[$(ii)$] $X^{-p}= \displaystyle\frac{\sin
p\,\pi}{p\,\pi}\int^\infty_0\!\!\!(\lambda\,I+X)^{-1}X(\lambda\,I+X)^{-1}\lambda^{-p}\,\emph{d}\lambda.$
\end{enumerate}
\end{lemma}

\begin{lemma}  \label{lem2}                                                             %Lemma 2
\cite{s27} There exists a unique positive definite solution $X$
of $X-A^*X^{-p}A=Q\;(0<p<1)$ and the iteration
\begin{equation}\label{eq10}
   X_{0}>0,\;\;X_{n}=Q+A^{*}X_{n-1}^{-p}A,\;\;n=1, 2, \cdots
\end{equation}
converges to $X$.
\end{lemma}

\begin{lemma}\label{lem8}                                        %%%%%%%%%%lemma3
\cite{s34}
\begin{enumerate}
 \item[$(i)$]If $X\in\mathcal{H}^{n\times n},$ then $\|e^{-X}\|=e^{-\lambda_{\min}(X)}.$
  \item [$(ii)$]If $X\in\mathcal{H}^{n\times n}$ and $r>0,$ then $X^{-r}=\frac{1}{\Gamma(r)}\int_{0}^{\infty}e^{-sX}s^{r-1}ds.$
  \item [$(ii)$]If $A, B\in\mathcal{C}^{n\times n},$ Then $e^{A+B}-e^{A}=\int_{0}^{1}e^{(1-t)A}Be^{t(A+B)}dt.$
\end{enumerate}
\end{lemma}
\section{A sufficient condition for the existence of a unique solution of $X-A^{*}X^{-p}A=Q\;(p>1)$}
In this section, we derive a new sufficient condition for the existence of a unique solution of $X-A^{*}X^{-p}A=Q\;(p>1)$ beginning with the lemma.

\begin{lemma}\label{lem6}
\cite{s40}If
 \begin{equation}\label{eq11}
\beta>(p\overline{\kappa})^{\frac{1}{p+1}},
\end{equation}
then Eq.(\ref{eq1}) has a unique positive definite solution $X\in [\beta I,\;\alpha I],$ where $\alpha$ and $\beta$ are respectively positive solutions of the following equations
$$
(x-\lambda_{max}(Q))\left(\lambda_{min}(Q)+\frac{\underline{\kappa}}{x^p}\right)^p=\overline{\kappa}
$$
and
$$                                                           %Eq.(2.10)
(x-\lambda_{min}(Q))\left(\lambda_{max}(Q)+\frac{\overline{\kappa}}{x^p}\right)^p=\underline{\kappa}.$$

Furthermore, \begin{equation}\label{eq13}
\lambda_{min}(Q)\leq\beta\leq\alpha.\end{equation}
\end{lemma}
\begin{theorem}
If
%$$(\left(p\lambda_{\max}(A^{*}A)\right)^{\frac{1}{p+1}}-\lambda_{\min}(Q))\left(\lambda_{\max}(Q)+\frac{\lambda_{\max}(A^{*}A)}
%{(p\lambda_{\max}(A^{*}A))\frac{p}{p+1}}\right)^{p}
%$$
\begin{equation}\label{eq8}
((p\overline{\kappa})^{\frac{1}{p+1}}-\lambda_{min}(Q))\left(\lambda_{max}(Q)+\frac{\overline{\kappa}}
{(p\overline{\kappa})^{\frac{p}{p+1}}}\right)^{p}<\underline{\kappa}\leq
\overline{\kappa}<\frac{\lambda_{max}(Q)\left(\lambda_{min}(Q)p\right)^{p}}{(p-1)^{p+1}},
\end{equation}
then Eq.(\ref{eq1}) has a unique positive definite solution.
\end{theorem}

\begin{proof}
We first prove $$\beta>(p\overline{\kappa})^{\frac{1}{p+1}}.$$

Let $$f(x)=(x-\lambda_{min}(Q))\left(\lambda_{max}(Q)+\frac{\overline{\kappa}}{x^{p}}\right)^{p}-\underline{\kappa}.$$
By computaiton, we obtain
$$f'(x)=\frac{\overline{\kappa}}{x^{p}}\left(\lambda_{max}(Q)+\frac{\overline{\kappa}}{x^{p}}\right)^{p-1}
\left(\frac{\lambda_{max}(Q)}{\overline{\kappa}}x^{p}+p^{2}\lambda_{min}(Q)x^{-1}+1-p^{2}\right).$$
Define that $$g(x)=\frac{\lambda_{max}(Q)}{\overline{\kappa}}x^{p}+p^{2}\lambda_{min}(Q)x^{-1}+1-p^{2}.$$ Then
$g(x)$ is decreasing on $[0,\;\left(\frac{\lambda_{min}(Q)p\overline{\kappa}}{\lambda_{max}(Q)}\right)^{\frac{1}{p+1}}]$
 and increasing on $[\left(\frac{\lambda_{min}(Q)p\overline{\kappa}}{\lambda_{max}(Q)}\right)^{\frac{1}{p+1}},\;+\infty),$
 which implies that $$g_{\min}=g\left(\left(\frac{\lambda_{min}(Q)p\overline{\kappa}}{\lambda_{max}(Q)}\right)^{\frac{1}{p+1}}\right)
 =(1+p)\left(\frac{(\lambda_{min}(Q)p)^{\frac{p}{p+1}}\lambda_{max}^{\frac{1}{p+1}}(Q)}{(\overline{\kappa})^{\frac{1}{p+1}}}
 +1-p\right).$$
 According to the condition $\overline{\kappa}<\frac{\lambda_{max}(Q)(\lambda_{min}(Q)p)^{p}}{(p-1)^{p+1}},$ it follows
 that $g_{\min}>0.$
 Noting that $$f'(x)=\frac{\overline{\kappa}}{x^{p}}\left(\lambda_{max}(Q)+\frac{\overline{\kappa}}{x^{p}}\right)^{p-1}g(x),$$
 which implies that $f(x)$ is increasing on $(0,\;+\infty).$
 Considering the condition (\ref{eq8}), one sees that $f((p\overline{\kappa})^{\frac{1}{p+1}})<0.$ Combining that and the definition of $\beta$ in Lemma \ref{lem6}, we obtain $\beta>(p\overline{\kappa})^{\frac{1}{p+1}}.$ By Lemma \ref{lem6}, Eq.(\ref{eq1}) has a unique positive definite solution.
\end{proof}

%%%%%%%%%%%%%%%%%%%%%%%%%section 3 Perturbation bound %%%%%%%%%%%%%%%%%%%%%%%%%%%%%%%%%
\section{Perturbation bound for $X-A^{*}X^{-p}A=Q\;(p>1)$}

Li and Zhang in \cite{s27} proved that there always exists a unique positive definite solution to the equation $X-A^*X^{-p}A=Q\;(0<p<1)$. They also obtained a perturbation bound for the unique solution.
But their approaches will become invalid for the case of $p>1.$ Since the equation $X-A^*X^{-p}A=Q\;(p>1)$ does not always have a unique positive definite solution, there are two difficulties for perturbation analysis to the equation $X-A^*X^{-p}A=Q\;(p>1)$. One difficulty is how to find some reasonable restrictions on the coefficient matrices of perturbed equation ensuring this equation has a unique positive definite solution. The other difficulty is how to find an expression of $\Delta X$ which is easy to handle.

Assume that the coefficient matrix $A$ is perturbed to $\widetilde{A}=\Delta A+A$.
Let $\widetilde{X}=\Delta X+X$ with $\Delta X\in \mathcal{H}^{n\times n}$ satisfying the perturbed equation
\begin{equation} \label{eq9}                                                                      %eq(4)
\widetilde{X}-\widetilde{A}^*\widetilde{X}^{-p}\widetilde{A}=Q,\;\;p>1.
\end{equation}

In the following, we derive a perturbation estimate for the positive definite solution to the
 matrix equation $X-A^*X^{-p}A=Q\;(p>1)$ beginning with the lemma.

%%%%%%%%%%%%%%%%%%%%%%%%%%%%%%%%%%%%%%%%%%%%%%%%%%%%%%%%%lemma
\begin{lemma}\label{lem7}
\cite{s40}
If$$p\|A\|^{2}<\lambda_{\min}^{p+1}(Q),$$ then Eq.(\ref{eq1}) has a unique positive definite solution $X,$ where $X\geq \lambda_{\min}(Q)I.$
\end{lemma}

%%%%%%%%%%%%%%%%%%%%%%%%%%%%%%%%%%%%%%%%%%%%%%%%%%%%%%theorem 4.1
\begin{theorem}\label{thm1}

If
 \begin{equation} \label{eq12}
\|A\|<\sqrt{\frac{\lambda_{\min}^{p+1}(Q)}{p}} \;\;\mbox{and}\;\;
\|\Delta A\|<\sqrt{\frac{\lambda_{\min}^{p+1}(Q)}{p}}-\|A\|,
\end{equation}
then
$$X-A^{*}X^{-p}A=Q\;\;\mbox{and}\;\;\widetilde{X}-\widetilde{A}^*\widetilde{X}^{-p}\widetilde{A}=Q$$
have  unique positive definite solutions $X$ and $\widetilde{X}$, respectively.
Furthermore, $$\frac{\|\widetilde{X}-X\|}{\|X\|}\leq\frac{(2\|A\|+\|\Delta A\|)}{\lambda_{\min}^{p+1}(Q)-p\|A\|^{2}}\|\Delta A\|\equiv\varrho.$$
\end{theorem}
\begin{proof}
By (\ref{eq12}), it follows that $\|\widetilde{A}\|\leq\|A\|+\|\Delta A\|\leq\sqrt{\frac{\lambda_{\min}^{p+1}(Q)}{p}}.$
According to Lemma \ref{lem7}, the condition (\ref{eq12}) ensures that Eq.(\ref{eq1}) and Eq.(\ref{eq9})
have unique positive definite solutions $X$ and $\widetilde{X},$ respectively. Furthermore, we obtain that
\begin{equation}\label{eq14}
X\geq \lambda_{\min}(Q)I,\;\;\;\widetilde{X}\geq \lambda_{\min}(Q)I.
\end{equation}
 Subtracting (\ref{eq9}) from (\ref{eq1}) gives
\begin{eqnarray}\label{eq15}
&& \Delta X=\widetilde{A}^{*}\widetilde{X}^{-p}\widetilde{A}-A^{*}X^{-p}A
 =A^{*}(\widetilde{X}^{-p}-X^{-p})A+\Delta A^{*}\widetilde{X}^{-p}A+\widetilde{A}^{*}
 \widetilde{X}^{-p}\Delta A.
 \end{eqnarray}\\
 By Lemma \ref{lem8} and inequalities in (\ref{eq14}), we have
 \begin{eqnarray}\label{eq27}
 &&\|\Delta X+A^{*}X^{-p}A-{A}^{*}\widetilde{X}^{-p}{A}\|\nonumber\\
  &=& \|\Delta X+A^{*}\frac{1}{\Gamma(p)}\int_{0}^{\infty}(e^{-sX}-e^{-s\widetilde{X}})s^{p-1}dsA\|\nonumber\\
  &=&\|\Delta X+A^{*}\frac{1}{\Gamma(p)}\int_{0}^{\infty}\int_{0}^{1}e^{-(1-t)s\widetilde{X}}\Delta X e^{-tsX}dts^{p}dsA\|\nonumber\\
  &\geq &\|\Delta X\|-\frac{\|A\|^{2}\|\Delta X\|}{\Gamma(p)}\int_{0}^{\infty}\int_{0}^{1}\|e^{-(1-t)s\widetilde{X}}\| \|e^{-tsX}\|dts^{p}ds\nonumber\\
  &\geq &\|\Delta X\|-\frac{\|A\|^{2}\|\Delta X\|}{\Gamma(p)}\int_{0}^{\infty}\int_{0}^{1}e^{-(1-t)s\lambda_{\min}(\widetilde{X})} e^{-ts\lambda_{\min}(X)}dts^{p}ds\nonumber\\
  &\geq &\|\Delta X\|-\frac{\|A\|^{2}\|\Delta X\|}{\Gamma(p)}\int_{0}^{\infty}\int_{0}^{1}e^{-(1-t)s\lambda_{\min}(Q)} e^{-ts\lambda_{\min}(Q)}dts^{p}ds\nonumber\\
  &=&\|\Delta X\|-\frac{\|A\|^{2}\|\Delta X\|}{\Gamma(p)}\int_{0}^{\infty}\int_{0}^{1}e^{-s\lambda_{\min}(Q)}dts^{p}ds\nonumber\\
  &=&\|\Delta X\|-\frac{\Gamma(p+1)}{\Gamma(p)}\cdot\frac{\|A\|^{2}\|\Delta X\|}{\lambda_{\min}^{p+1}(Q)}\nonumber\\
  &=&\frac{\lambda_{\min}^{p+1}(Q)-p\|A\|^{2}}{\lambda_{\min}^{p+1}(Q)}\|\Delta X\|.
 \end{eqnarray}
 Noting (\ref{eq12}), we have $$\lambda_{\min}^{p+1}(Q)-p\|A\|^{2}>0.$$
 Combining (\ref{eq15}) and (\ref{eq27}), one sees that
 \begin{eqnarray*}
 \frac{\lambda_{\min}^{p+1}(Q)-p\|A\|^{2}}{\lambda_{\min}^{p+1}(Q)}\|\Delta X\|&\leq &\|\Delta A^{*}\widetilde{X}^{-p}A+\widetilde{A}^{*}
 \widetilde{X}^{-p}\Delta A\|
 \leq (\|\Delta A\|+2\|A\|)\|\Delta A\|
 \|\widetilde{X}^{-p}\|\\
 &\leq& (\|\Delta A\|+2\|A\|)\|\Delta A\|\lambda_{\min}^{-p}(Q),
\end{eqnarray*}
which implies that
 \begin{eqnarray*}
 \frac{\|\Delta X\|}{\|X\|}
 &\leq& \frac{(\|\Delta A\|+2\|A\|)}{\lambda_{\min}^{p+1}(Q)-p\|A\|^{2}}\|\Delta A\|. \\
\end{eqnarray*}
 \end{proof}

 %%%%%%%%%%%%%%%%%%%%%%%%%section 3 condition number p>1 %%%%%%%%%%%%%%%%%%%%%%%%%%%%%%%%%
\section{Condition number for $X-A^{*}X^{-p}A=Q\;(p>1)$}

A condition number is a measurement of the sensitivity of the positive definite stabilizing solutions to
small changes in the coefficient matrices.
In this section, we apply the theory of condition number developed by Rice \cite{s37} to
derive explicit expressions of the condition number for the matrix equation $X-A^{*}X^{-p}A=Q\;(p>1).$

Here we consider the perturbed equation
\begin{equation} \label{eq23}                                                                    %eq(4)
\widetilde{X}-\widetilde{A}^*\widetilde{X}^{-p}\widetilde{A}=\widetilde{Q},\;\;p>1,
\end{equation}
where $\widetilde{A}$ and $\widetilde{Q}$ are small perturbations
of $A$ and $Q$ in Eq.(\ref{eq1}), respectively.

Suppose that
$
p\|A\|^{2}<\lambda_{\min}^{p+1}(Q) \;\;\mbox{and}\;\;
p\|\widetilde{A}\|^{2}<\lambda_{\min}^{p+1}(\widetilde{Q}).
$
According to Lemma \ref{lem7}, Eq.(\ref{eq1}) and
Eq.(\ref{eq23}) have unique positive definite solutions $X$
and $\widetilde{X}$, respectively. Let $\Delta X=\widetilde{X}-X$, $\Delta
Q=\widetilde{Q}-Q$ and $\Delta A=\widetilde{A}-A$.

 Subtracting (\ref{eq23}) from (\ref{eq1}) gives

 \begin{eqnarray*}
 \Delta X&=&\widetilde{A}^{*}\widetilde{X}^{-p}\widetilde{A}-A^{*}X^{-p}A+\Delta Q
 =A^{*}(\widetilde{X}^{-p}-X^{-p})A+\Delta A^{*}\widetilde{X}^{-p}A+\widetilde{A}^{*}
 \widetilde{X}^{-p}\Delta A+\Delta Q\\
 &=&-A^{*}\frac{1}{\Gamma(p)}\int_{0}^{\infty}(e^{-sX}-e^{-s\widetilde{X}})s^{p-1}d s A
 +\Delta A^{*}\widetilde{X}^{-p}A+\widetilde{A}^{*}
 \widetilde{X}^{-p}\Delta A+\Delta Q\\
 &=&-A^{*}\frac{1}{\Gamma(p)}\int_{0}^{\infty}\int_{0}^{1}e^{-(1-t)s\widetilde{X}}
 (\widetilde{X}-X)e^{-tsX}dts^{p}d s A +\Delta A^{*}\widetilde{X}^{-p}A+\widetilde{A}^{*}
 \widetilde{X}^{-p}\Delta A+\Delta Q\\
 &=&-A^{*}\frac{1}{\Gamma(p)}\int_{0}^{\infty}\int_{0}^{1}(e^{-(1-t)s\widetilde{X}}
 -e^{-(1-t)sX})
 \Delta X e^{-tsX}dts^{p}d s A+\Delta Q\\
 &&-A^{*}\frac{1}{\Gamma(p)}\int_{0}^{\infty}\int_{0}^{1}e^{-(1-t)sX}
 \Delta X e^{-tsX}dts^{p}d s A
 -(\widetilde{A}^{*}X^{-p}\Delta A-\widetilde{A}^{*}(X+\Delta X)^{-p}\Delta A)\\
 &&+\widetilde{A}^{*}X^{-p}\Delta A-
 (\Delta A^{*}X^{-p}A-\Delta A^{*}(X+\Delta X)^{-p}A)+\Delta A^{*}X^{-p}A\\
 &=&A^{*}\frac{1}{\Gamma(p)}\int_{0}^{\infty}\int_{0}^{1}\int_{0}^{1}
 e^{-(1-m)(1-t)sX}\Delta X e^{-m(1-t)s\widetilde{X}}\Delta Xe^{-tsX}
 d m(1-t)dts^{p+1}dsA+\Delta Q\\
 &&-A^{*}\frac{1}{\Gamma(p)}\int_{0}^{\infty}\int_{0}^{1}e^{-(1-t)sX}
 \Delta X e^{-tsX}dts^{p}d s A+\Delta A^{*}X^{-p}\Delta A+A^{*}X^{-p}\Delta A+
 \Delta A^{*}X^{-p}A\\
 &&-\widetilde{A}^{*}\frac{1}{\Gamma(p)}\int_{0}^{\infty}\int_{0}^{1}e^{-(1-t)s(X+\Delta X)}
 \Delta X e^{-tsX}dts^{p}ds\Delta A\\
 &&-\Delta A^{*}\frac{1}{\Gamma(p)}\int_{0}^{\infty}\int_{0}^{1}e^{-(1-t)s(X+\Delta X)}
 \Delta X e^{-tsX}dts^{p}d s A.
 \end{eqnarray*}
 Therefore
 \begin{equation}    \label{eq18}                                        %方程（4）
 \Delta X+A^{*}\frac{1}{\Gamma(p)}\int_{0}^{\infty}\int_{0}^{1}e^{-(1-t)sX}
 \Delta X e^{-tsX}dts^{p}dsA=E+h(\Delta X),
 \end{equation}
 where
 \begin{eqnarray*}
 &&B=X^{-p}A,\nonumber\\
 &&E=\Delta Q+(B^{*}\Delta A+\Delta A^{*}B)+\Delta A^{*}X^{-p}\Delta A,\nonumber\\
 &&h(\Delta X)=A^{*}\frac{1}{\Gamma(p)}\int_{0}^{\infty}\int_{0}^{1}\int_{0}^{1}
 e^{-(1-m)(1-t)sX}\Delta X e^{-m(1-t)s\widetilde{X}}\Delta Xe^{-tsX}
 d m(1-t)dts^{p+1}dsA\\
 &&{\hspace{1.15cm}}-\widetilde{A}^{*}\frac{1}{\Gamma(p)}\int_{0}^{\infty}\int_{0}^{1}e^{-(1-t)s(X+\Delta X)}
 \Delta X e^{-tsX}dts^{p}ds\Delta A\nonumber\\
 &&{\hspace{1.15cm}}-\Delta A^{*}\frac{1}{\Gamma(p)}\int_{0}^{\infty}\int_{0}^{1}e^{-(1-t)s(X+\Delta X)}
 \Delta X e^{-tsX}dts^{p}d s A.\nonumber
 \end{eqnarray*}

 \begin{lemma}  \label{lem5}                                                %lemma(5)
If
\begin{equation}\label{eq16}
p\|A\|^{2}<\lambda_{\min}^{p+1}(Q),
\end{equation}
 then the linear operator $\mathbf{V}:\mathcal{H}^{n\times
n}\rightarrow\mathcal{H}^{n\times n}$ defined by
\begin{equation}\label{eq17}
 \mathbf{V}W=W+\frac{1}{\Gamma(p)}\int_{0}^{\infty}\int_{0}^{1}A^{*}e^{-(1-t)sX}W
 e^{-tsX}Adts^{p}ds, \;\;\;W\in\mathcal{H}^{n\times n}.
\end{equation}
is invertible.
\end{lemma}
\begin{proof}
 Define the operator
$\mathbf{R}:\mathcal{H}^{n\times
n}\rightarrow\mathcal{H}^{n\times n} $ by
$$ \mathbf{R}Z=\!\frac{1}{\Gamma(p)}\int^{\infty}_{0}\int_{0}^{1}
A^{*}e^{-(1-t)sX}Z
 e^{-tsX}Adts^{p}ds,
\;Z\in\mathcal{H}^{n\times n}\!\!,
$$
it follows that
$$
\mathbf{V}W=W+\mathbf{R}W.$$
Then $\mathbf{V}$ is invertible if and only if $I+\mathbf{R}$ is invertible.

According to Lemma \ref{lem8} and the condition (\ref{eq16}), we
have
\begin{eqnarray*}
||\mathbf{R}W||&\leq &||A||^{2}||W||\frac{
1}{\Gamma(p)}\!\int^{\infty}_{0}\int_{0}^{1}||e^{-(1-t)sX}||||e^{-tsX}||dts^{p}ds\\
&= &||A||^{2}||W||||\frac{1}{\Gamma(p)}\!\int^{\infty}_{0}\int_{0}^{1}e^{-(1-t)s\lambda_{min}
(X)}e^{-ts\lambda_{min}
(X)}dts^{p}ds\\
&\leq &||A||^{2}||W||||\frac{1}{\Gamma(p)}\!\int^{\infty}_{0}\int_{0}^{1}e^{-(1-t)s\lambda_{min}(Q)}
e^{-ts\lambda_{min}(Q)}dts^{p}ds\\
& =&||A||^{2}||W||\frac{1}{\Gamma(p)}\int^{\infty}_{0}e^{-s\lambda_{min}(Q)}s^{p}ds\\
&=&\frac{p||A||^{2}}{\lambda_{min}^{p+1}(Q)}||W||
<||W||,
\end{eqnarray*}
which implies that $||\mathbf{R}||<1$ and $I+\mathbf{R}$ is
invertible. Therefore, the operator $\mathbf{V}$ is
invertible.
\end{proof}
Thus, we can rewrite (\ref{eq18}) as
$$\Delta X=\mathbf{V}^{-1}\Delta Q+\mathbf{V}^{-1}(B^{*}\Delta A+\Delta A^{*}B)+\mathbf{V}^{-1}
(\Delta A^{*}X^{-p}\Delta A)+\mathbf{V}^{-1}(h(\Delta X)).$$
Obviously,
\begin{equation}    \label{eq21}                                                             %方程98
\Delta X=\mathbf{V}^{-1}\Delta Q+\mathbf{V}^{-1}(B^{*}\Delta A+\Delta A^{*}B)+O(||(\Delta A,\;\Delta Q)||_{F}^{2}),
(\Delta A,\;\Delta Q)\rightarrow 0.
\end{equation}

By the theory of condition number developed by Rice \cite{s19}, we define
the condition number of the Hermitian positive definite solution
$X$ to the matrix equation $X-A^{*}X^{-p}A=Q\;(p>1)$ by
\begin{equation}\label{eq22}
c(X)=\lim_{\delta\rightarrow 0}\sup_{||(\frac{\Delta A}{\eta},
\frac{\Delta Q}{\rho})||_{F}\leq\delta}\frac{||\Delta
X||_{F}}{\xi\delta},
\end{equation}

where $\xi$, $\eta$ and $\rho$ are positive parameters. Taking
$\xi=\eta=\rho=1$  in (\ref{eq22}) gives the absolute condition number $c_{abs}(X)$,
and taking $\xi=||X||_{F}$, $\eta=||A||_{F}$  and $\rho=||Q||_{F}$ in (\ref{eq22})
gives the relative condition number $c_{rel}(X)$.

Substituting (\ref{eq21}) into (\ref{eq22}), we get
\begin{eqnarray*}
c(X)&=&\frac{1}{\xi}\!\!\!\max_{\begin{array}{c}(\frac{\Delta
A}{\eta},
\frac{\Delta Q }{\rho})\neq 0\\
 \Delta A\in\mathcal{C}^{n\times
n}, \Delta Q\in\mathcal{H}^{n\times n}
\end{array}}
\!\!\!\!\!\!\!\!\frac{||\mathbf{V}^{-1}(\Delta Q+B^{*}\Delta
A+\Delta A^{*}B) ||_{F}}{||(\frac{\Delta A}{\eta}, \frac{\Delta
Q}{\rho})||_{F}}\nonumber\\
&=&\frac{1}{\xi}\!\!\!\max_{\begin{array}{c}(E,
H)\neq 0\\
 E\in\mathcal{C}^{n\times
n}, H\in\mathcal{H}^{n\times n}
\end{array}}
\!\!\!\!\!\!\!\!\frac{||\mathbf{V}^{-1}(\rho H+\eta(B^{*}E+E^{*}B))
||_{F}}{||(E, H)||_{F}}.
\end{eqnarray*}
Let $V$ be the matrix representation of the linear operator
$\mathbf{V}$. Then it is easy to see that
\begin{equation}    \label{eq25}                                                %方程（9）
V=I\otimes I+\!\frac{1}{\Gamma(p)}\int^{\infty}_{0}\int_{0}^{1}
(e^{-tsX}A)^{T}\otimes (A^{*}e^{-(1-t)sX})dts^{p}ds.
\end{equation}
Let
\begin{eqnarray}\label{eq24}
&&V^{-1}=S+i\Sigma,\nonumber\\
&&V^{-1}(I\otimes B^{*})=V^{-1}(I\otimes
(X^{-p}A)^{*})=U_{1}+i\Omega_{1},\\
&&V^{-1}(B^{T}\otimes I)\Pi=V^{-1}((X^{-p}A)^{T}\otimes
I)\Pi=U_{2}+i\Omega_{2},\nonumber
\end{eqnarray}
\begin{equation}
 S_{c}=\left[\begin{array}{cc}\label{eq28}                                      %16
S & -\Sigma\\
\Sigma & S
\end{array}\right],\;\;\;\;                                          %方程（11）
U_{c}=\left[\begin{array}{cc}
U_{1}+U_{2} & \Omega_{2}-\Omega_{1}\\
\Omega_{1}+\Omega_{2} & U_{1}-U_{2}
\end{array}\right],
\end{equation}

$$\vector H=x+\textbf{i}y,\;\;\vector E=a+\textbf{i}b,
\;\;g=(x^{T}, y^{T}, a^{T}, b^{T})^{T},$$
where $x, y, a, b \in\mathcal{R}^{n^{2}},\; S, \Sigma, U_{1}, U_{2}, \Omega_{1}, \Omega_{2}\in\mathcal{R}^{n^{2}\times n^{2}},M=(E,H),$\;$\textbf{i}=\sqrt{-1},$ $\Pi$ is the vec-permutation matrix, i.e.,
$$\vector \;E^{T}=\Pi \;\vector \;E.$$

Furthermore, we obtain that
\begin{eqnarray*}
&&c(X)=\frac{1}{\xi}\max_{\begin{array}{c}M\neq 0\\
\end{array}}
\frac{||\mathbf{V}^{-1}(\rho H+\eta(B^{*}E+E^{*}B))
||_{F}}{||(E, H)||_{F}}\\
&=&\frac{1}{\xi}\max_{\begin{array}{c}M\neq 0\\
\end{array}}
\frac{||\rho{V}^{-1} \vector H+\eta{V}^{-1}((I\otimes B^{*})\vector E+(B^{T}\otimes I)\vector E^{*})
||}{\left\|\left(
\vector E, \vector H
\right)\right\|}\\
&=&\frac{1}{\xi}\!\!\max_{\begin{array}{c}M\neq 0\\
\end{array}}
\!\!\!\!\!\frac{||\rho(S+\textbf{i}\Sigma)(x+\textbf{i}y)+
\eta[(U_{1}+\textbf{i}\Omega_{1})(a+\textbf{i}b)
+(U_{2}+\textbf{i}\Omega_{2})(a-\textbf{i}b)]
||}{\left\|\left(
\vector E, \vector H
\right)\right\|}\\
&=&\frac{1}{\xi}\!\!\!\max_{\begin{array}{c}g\neq 0\\
\end{array}}
\frac{||(\rho\; S_{c}, \eta U_{c})g
||}{\|g\|}\\
&=&\frac{1}{\xi}\;||\;(\rho S_{c},\;\eta U_{c} )||,\;\;E\in\mathcal{C}^{n\times
n}, H\in\mathcal{H}^{n\times n}.
\end{eqnarray*}

Then we have the following theorem.
\begin{theorem}\label{thm6}                                           %theorem6
If $p\|A\|^{2}<\lambda_{\min}^{p+1}(Q)$, then the condition number $c(X)$ defined by (\ref{eq22}) has the
explicit expression
\begin{equation}\label{eq17}                                        %eq17
c(X)=\frac{1}{\xi}\;||\;(\rho S_{c},\;\;\eta U_{c})\;||,
\end{equation}
where the matrices $S_{c}$ and $U_{c}$ are defined by
(\ref{eq25})$-$(\ref{eq28}).
\end{theorem}
\begin{rem}\label{rem1}                                            %remark1
From (\ref{eq17}) we have the relative condition number
\begin{equation} \label{eq26}                                             %18
c_{rel}(X)=\frac{||\;(||Q||_{F}S_{c},\;\;||A||_{F}
U_{c})\;||}{||X||_{F}}.
\end{equation}
\end{rem}
%%%%%%%%%%%%%%%%%%%%%%%%%%%%%5.2%%%%%%%%%%%%%%%%%%%%%%%%%%%%%%%%%%%%%%%%
\subsection{The real case  }
In this subsection we consider the real case, i.e.,
all the coefficient matrices $A$, $Q$ of the matrix equation $X-A^{*}X^{-p}A=Q\;(p>1)$ are real. In
such a case the corresponding solution $X$ is also real.
Completely similar arguments as in Theorem \ref{thm6} give the following
theorem.
\begin{theorem}
Let $A$, $Q$ be real, $c(X)$ be the condition number defined by
(\ref{eq22}). If $p\|A\|^{2}<\lambda_{\min}^{p+1}(Q),$ then $c(X)$ has the explicit expression
\begin{equation*}
c(X)=\frac{1}{\xi}\;||\;(\rho S_{r},\;\;\eta U_{r})\;||,
\end{equation*}
where
\begin{eqnarray*}
&&S_{r}=\left(I\otimes I+\!\frac{1}{\Gamma(p)}\int^{\infty}_{0}\int_{0}^{1}
(e^{-tsX}A)^{T}\otimes (A^{*}e^{-(1-t)sX})dts^{p}ds\right)^{-1},\\
&&U_{r}=S_{r}[I\otimes(A^{T}X^{-p})+((A^{T}X^{-p})\otimes I)\Pi].
\end{eqnarray*}
\end{theorem}

\begin{rem}\label{rem2}                                              %rem2
In the real case the relative condition number is given by
\begin{equation*}
c_{rel}(X)=\frac{||\;(||Q||_{F}S_{r},\;\;||A||_{F}
U_{r})\;||}{||X||_{F}}.
\end{equation*}
\end{rem}
%%%%%%%%%%%%%%%%%%%%%%%%%section 3 Perturbation bound %%%%%%%%%%%%%%%%%%%%%%%%%%%%%%%%%
\section{New perturbation bound for $X-A^{*}X^{-p}A=Q\;(0<p<1)$}
Here we consider the perturbed equation
\begin{equation} \label{eq4}                                                                      %eq(4)
\widetilde{X}-\widetilde{A}^*\widetilde{X}^{-p}\widetilde{A}=\widetilde{Q}, \;\;0<p<1,
\end{equation}
where $\widetilde{A}$ and $\widetilde{Q}$ are small perturbations
of $A$ and $Q$ in Eq.(\ref{eq1}), respectively. We assume that $X$
and $\widetilde{X}$ are the solutions of Eq.(\ref{eq1}) and
Eq.(\ref{eq4}), respectively. Let $\Delta X=\widetilde{X}-X$, $\Delta
Q=\widetilde{Q}-Q$ and $\Delta A=\widetilde{A}-A$.

In this section we develop a new perturbation bound for the
solution of Eq.(\ref{eq1}) which is sharper than that in Theorem 3.1 \cite{s27}.

Subtracting Eq.(\ref{eq1}) from Eq.(\ref{eq4}), using Lemma \ref{lem1}, we have
\begin{equation}\label{eq6}
\Delta X+\frac{\sin p\pi}{\pi}\int^{\infty}_{0}[(\lambda I+X
)^{-1}A]^{*}\Delta X[(\lambda I+X
)^{-1}A]\lambda^{-p}d\lambda=E+h(\Delta X),
\end{equation}
where
\begin{eqnarray}\label{eq7}
\!\!&&B=X^{-p}A,\nonumber\\
\!\!&&E=\Delta Q+(B^{*}\Delta A+\Delta A^{*}B)+\Delta
A^{*}X^{-p}\Delta A,\\
\!\!&&h(\Delta X)=\frac{\sin
p\pi}{\pi}A^{*}\!\!\!\int^{\infty}_{0}\!\!\!\!(\lambda I\!+\!X
)^{-1}\!\Delta X(\lambda I\!+\!X\!+\!\Delta X)^{-1}\!\Delta
X(\lambda I\!+\!X )^{-1}\lambda^{-p}\textmd{d}\lambda A\nonumber\\
\!\!&&{\hspace{1.15cm}}-\frac{\sin
p\pi}{\pi}\widetilde{A}^{*}\!\!\!\int^{\infty}_{0}\!\!\!\!(\lambda
I\!+\!X )^{-1}\!\Delta X(\lambda I\!+\!X\!+\!\Delta
X)^{-1}\lambda^{-p}\textmd{d}\lambda\Delta A\nonumber\\
\!\!&&{\hspace{1.15cm}}-\frac{\sin p\pi}{\pi}\Delta
A^{*}\!\!\!\int^{\infty}_{0}\!\!\!\!(\lambda I\!+\!X
)^{-1}\!\Delta X(\lambda I\!+\!X\!+\!\Delta
X)^{-1}\lambda^{-p}\textmd{d}\lambda A\nonumber.
\end{eqnarray}
By Lemma 5.1 in \cite{s27},  the linear operator $\mathbf{L}:\mathcal{H}^{n\times
n}\rightarrow\mathcal{H}^{n\times n}$ defined by
\begin{equation*}
 \mathbf{L}W=W+\frac{\sin p\pi}{\pi}\int^{\infty}_{0}[(\lambda I+X
)^{-1}A]^{*}W[(\lambda I+X )^{-1}A]\lambda^{-p}\emph{d}\lambda,
\;\;\;W\in\mathcal{H}^{n\times n}.
 \end{equation*}
  is invertible.

We also define operator $ {\textbf P}:\mathcal{C}^{n\times n}\rightarrow\mathcal{H}^{n\times n}$ by

$${\textbf P}Z=\textbf{L}^{-1}(B^{*}Z+Z^{*}B),\;\;Z\in\mathcal{C}^{n\times n},\;\;i=1,2, \cdots, m.$$
Thus,we can rewrite (\ref{eq6}) as
\begin{equation}\label{eq19}                                                          %eq(4.11)
    \Delta X=\textbf{L}^{-1}\Delta Q+{\textbf P}\Delta A+
\textbf{L}^{-1}(\Delta A^{*}X^{-p}\Delta A)+\textbf{L}^{-1}(h(\Delta X)).
\end{equation}
Define
$$
||\mathbf{L}^{-1}||=\max_{\begin{array}{c}
W\in\mathcal{H}^{n\times n}\\
||W||=1
\end{array}}||\mathbf{L}^{-1}W||,\;\;\;\;
||\mathbf{P}||=\max_{\begin{array}{c}
Z\in\mathcal{C}^{n\times n}\\
||Z||=1
\end{array}}||\mathbf{P}Z||.
$$
\iffalse
Let $L$ and $P_{i}$ be the matrix representation of the operators $\textbf{L}$ and ${\textbf P_{i}},$ respectively.
Then it follows that
\begin{eqnarray*}
 L&=& I+\sum\limits_{i=1}^{m}(X^{-1}A_{i})^{T}\otimes (X^{-1}A_{i})^{*},\\
 P_{i}&=&L^{-1}(I_{n}\otimes (X^{-1}A_{i})^{*}+((X^{-1}A_{i})^{T}\otimes I_{n})\Pi),\;\;i=1,2, \cdots, m.
 \end{eqnarray*}
 \fi
 Now we denote
 \begin{eqnarray*}
l&=&\|\textbf{L}^{-1}\|^{-1},\;\;\zeta=\|X^{-1}\|,\;\;\xi=\|X^{-p}\|,\;\;n=\|\textbf{P}\|,\;\;
\eta=p\xi\|A\|^{2}\\
\epsilon&=& \frac{1}{l}\|\Delta Q\|+n\|\Delta A\|+\frac{\xi}{l}\|\Delta A\|^{2},\;\;\;\;\sigma\;\;=\;\;\frac{p}{l}\zeta\xi(2\|A\|+\|\Delta A\|)\|\Delta A\|.
 \end{eqnarray*}

  \begin{theorem}\label{thm9}                                         %theorem 4.3
  If
  \begin{equation}\label{eq20}                                             %eq4.12
   \sigma<1\;\;\mbox{and}\;\;\epsilon<\frac{l(1-\sigma)^{2}}
   {\zeta(l+l\sigma+2\eta+2\sqrt{(l\sigma+\eta)(\eta+l)})},
  \end{equation}
  then
  \begin{equation*}
    \|\widetilde{X}-X\|\leq\frac{2l\epsilon}{l(1+\zeta\epsilon-\sigma)+
    \sqrt{l^{2}(1+\zeta\epsilon-\sigma)^{2}-4l\zeta\epsilon(l+\eta)}}\equiv\mu_{*}
  \end{equation*}

 \end{theorem}

   \begin{proof}
   {Let $$f(\Delta X)=\textbf{L}^{-1}\Delta Q+{\textbf P}\Delta A+
\textbf{L}^{-1}(\Delta A^{*}X^{-p}\Delta A)+\textbf{L}^{-1}(h(\Delta X)).$$ Obviously, $f:\mathcal{H}^{n\times n}\rightarrow\mathcal{H}^{n\times n}$ is continuous. The condition (\ref{eq20}) ensures that the quadratic equation $\zeta(l+\eta)x^{2}-l(1+\zeta\epsilon-\sigma)x+l\epsilon=0$ in $x$ has two positive real roots.
The smaller one is
$$\mu_{*}=\frac{2l\epsilon}{l(1+\zeta\epsilon-\sigma)+
    \sqrt{l^{2}(1+\zeta\epsilon-\sigma)^{2}-4l\zeta\epsilon(l+\eta)}}.$$ Define
    $\Omega=\{\Delta X\in \mathcal{H}^{n\times n}: \|\Delta X\|\leq \mu_{*}\}.$ Then for any $\Delta X\in \Omega,$ by (\ref{eq20}), we have
    \begin{eqnarray*}
   && ||X^{-1}\Delta X||\leq ||X^{-1}||||\Delta X||\leq\zeta\;\mu_{*}\leq
\zeta\cdot\frac{2\epsilon}{1+\epsilon-\sigma}\\
&&=
1+\frac{\zeta\epsilon+\sigma-1}{1+\zeta\epsilon-\sigma}\leq 1+\frac{-2(1-\sigma)(l\sigma+\eta)}{(l\sigma+l+2\eta)(1+\zeta\epsilon-\sigma)}<1.
\end{eqnarray*}
It follows that $I-X^{-1}\Delta X$ is nonsingular and
$$\|I-X^{-1}\Delta X\|\leq\frac{1}{1-\|X^{-1}\Delta X\|}\leq\frac{1}{1-\zeta\|\Delta X\|}.$$
Therefore
\begin{eqnarray*}
\|f(\Delta X)\|&\leq& \frac{1}{l}\|\Delta Q\|+n\|\Delta A\|+\frac{\xi}{l}\|\Delta A_{i}\|^{2}+\frac{p}{l}\zeta\xi\|A\|^{2}\frac{\|\Delta X\|^{2}}{1-\zeta\|\Delta X\|}\\
&+&\frac{p}{l}\zeta\xi(2\|A\|+\|\Delta A\|)\|\Delta A\|\cdot\frac{\|\Delta X\|}{1-\zeta\|\Delta X\|}\\
&\leq &\epsilon+\frac{\sigma\|\Delta X\|}{1-\zeta\|\Delta X\|}+\frac{\eta\zeta\|\Delta X\|^{2}}{l(1-\zeta\|\Delta X\|)}\\
&\leq &\epsilon+\frac{\sigma\mu_{*}}{1-\zeta\mu_{*}}+\frac{\theta\zeta\mu_{*}^{2}}{l(1-\zeta\mu_{*})}=\mu_{*},
\end{eqnarray*}
for $\Delta X\in \Omega.$ That is $f(\Omega)\subseteq\Omega.$ According to Schauder fixed point theorem, there exists $\Delta X_{*}\in\Omega$ such that $f(\Delta X_{*})=\Delta X_{*}.$ It follows that $X+\Delta X_{*}$ is a Hermitian solution of Eq.(\ref{eq4}). By Lemma \ref{lem2}, we know that the solution of Eq.(\ref{eq4}) is unique. Then $\Delta X_{*}=\widetilde{X}-X$ and $\|\widetilde{X}-X\|\leq\mu_{*}.$
}\end{proof}

%%%%%%%%%%%%%%%%section 4 Backward error%%%%%%%%%%%%%%%%%%%%%%%%%%%%%%%%%%%%%%%%%%%%%%%%%%%%%%%%%
\section{New backward error for $X-A^{*}X^{-p}A=Q\;(0<p<1)$}
In this section we evaluate a new backward error of an approximate
solution to the unique solution, which is sharper than that in Theorem 4.1 \cite{s27}.
\begin{theorem}\label{thm5}                                                             %theorem 5
Let $\widetilde{X}>0$ be an approximation to the solution $X$ of
(\ref{eq1}). If $\|\widetilde{X}^{-\frac{p}{2}}A\|^{2}\|\widetilde{X}^{-1}\|<1$ and the residual $R(\widetilde{X}) \equiv
Q+A^*\widetilde{X}^{-p}A-\widetilde{X}$ satisfies
\begin{equation}\label{eq2}
    \|R(\widetilde{X})\| \leq\frac{\theta_{1}}{2}
    \min\left\{1,\frac{\theta_{1}}{2\lambda_{\min}(\widetilde{X})}\right\},\;\mbox{where}\;\;\;
    \theta_{1}=(1-\|\widetilde{X}^{-\frac{p}{2}}A\|^{2}\|\widetilde{X}^{-1}\|)
    \lambda_{\min}(\widetilde{X})+\|R(\widetilde{X})\|>0,
\end{equation}
then
\begin{equation}\label{eq3}
    \| \widetilde{X}-X\|
    \leq\theta\|R(\widetilde{X})\|,\;\;\mbox{where}\;\;\;\theta=\frac{2\lambda_{\min}(\widetilde{X})}
    {\theta_{1}+\sqrt{\theta_{1}^{2}-4\lambda_{\min}(\widetilde{X})\|R(\widetilde{X})\|}}.
\end{equation}
\end{theorem}

To prove the above theorem, we first verify the following lemma.
\begin{lemma}  \label{lem5}                                                             %Lemma 4
For every positive definite matrix $X\in\mathcal{H}^{n\times n}$,
$0<p<1$, if $X+\Delta X\geq (1/\nu) I>0,$ then
\begin{equation}  \label{eq5}                                                                %eq(5)
 \| A^*((X+\Delta
X)^{-p}-X^{-p})A\| \leq p\,(\|\Delta X \| +
\nu\|\Delta X\|^{2})\|X^{-\frac{p}{2}}A\|^{2}\|X^{-1}\|.
\end{equation}
\end{lemma}
\begin{proof}
If $X+\Delta X\geq (1/\nu) I>0$, then
\begin{eqnarray*}
& &\| A^*((X+\Delta X)^{-p}-X^{-p})A\|\\
&=&\left\| A^*\left(\frac{\sin
p\,\pi}{\pi}\int^\infty_0\left((\lambda\,I+X+\Delta X)^{-1}\nonumber  -
(\lambda\,I+X)^{-1}\right)\lambda^{-p}\,\textmd{d}\lambda\right)
A\right\|\\
&\leq &\frac{\sin
p\,\pi}{\pi}\left(\|
A^*\int^\infty_0(\lambda\,I+X)^{-1}\Delta
X(\lambda\,I+X)^{-1}\lambda^{-p}\textmd{d}\lambda\;A \|\right)\\
&&+\frac{\sin
p\,\pi}{\pi}\left(\|A^*\int^\infty_0\,(\lambda\,I+X )^{-1}\Delta
X(\lambda\,I+X+\Delta X
)^{-1}\Delta X(\lambda\,I+X)^{-1}\lambda^{-p}\textmd{d}\lambda\;A \|\right)\\
& \leq & p\,\|A^*X^{-p}A\|\|X^{-1}\|\| \Delta X\|+p\,\|A^*X^{-p}A\|\nu{\|\Delta X\|}^{2}\|X^{-1}\|\\
& = &p\,(\|\Delta X \| +
\nu\|\Delta X\|^{2})\|X^{-\frac{p}{2}}A\|^{2}\|X^{-1}\|.
\end{eqnarray*}
\end{proof}

\begin{proof}
Let $$\Psi = \{\Delta X \in
\mathcal{H}^{n\times n}:\|\Delta X\|\leq \theta
\|R(\widetilde{X})\|\}.$$ Obviously, $\Psi$ is a nonempty
bounded convex closed set. Let $$g(\Delta
X)=A^*((\widetilde{X}+\Delta
X)^{-p}-\widetilde{X}^{-p})A+R(\widetilde{X}).$$ Evidently $g:
\Psi\mapsto\mathcal{H}^{n\times n}$ is continuous. We will prove
that $g(\Psi)\subseteq\Psi$. For every $\Delta X\in\Psi,$ we have $$\Delta X\geq -\theta\|R(\widetilde{X})\|I.$$ Hence
 $$\widetilde{X}+\Delta X\geq \widetilde{X}-\theta\|R(\widetilde{X})\|I\geq (\lambda_{\min}(\widetilde{X})-\theta\|R(\widetilde{X})\|)I.$$
  Using (\ref{eq2}) and (\ref{eq3}), one sees
that
$$\theta\|R(\widetilde{X})\|=\frac{2\lambda_{\min}(\widetilde{X})\|R(\widetilde{X})\|}
    {\theta_{1}+\sqrt{\theta_{1}^{2}-4\lambda_{\min}(\widetilde{X})\|R(\widetilde{X})\|}}<
\frac{2\lambda_{\min}(\widetilde{X})\|R(\widetilde{X})\|}{\theta_{1}}<\lambda_{\min}(\widetilde{X}).$$
Therefore, $(\lambda_{\min}(\widetilde{X})-\theta\|R(\widetilde{X})\|)I>0.$

According to (\ref{eq5}), we obtain
\begin{eqnarray*}
&&\|g(\Delta X)\|\\
&\leq&p\,(\|\Delta X \| +
\frac{\|\Delta X\|^{2}}{\lambda_{\min}(\widetilde{X})-\theta\|R(\widetilde{X})\|})\|X^{-\frac{p}{2}}A\|^{2}\|\widetilde{X}^{-1}\|+\|R(\widetilde{X})\|\\
&\leq &\left(\theta\|R(\widetilde{X})\|+\frac{(\theta\|R(\widetilde{X})\|)^{2}}
{\lambda_{\min}(\widetilde{X})-\theta\|R(\widetilde{X})\|}\right)(p\|X^{-\frac{p}{2}}A\|^{2}\|\widetilde{X}^{-1}\|)
+\|R(\widetilde{X})\|\\
&=&\theta\|R(\widetilde{X})\|.
\end{eqnarray*}
By Brouwer's fixed point theorem, there exists a $\Delta X\in\Psi$
such that $g(\Delta X)=\Delta X.$ Hence $\widetilde{X}+\Delta X$
is a solution of Eq.(\ref{eq1}). Moreover, by Lemma \ref{lem2}, we
know that the solution $X$ of Eq.(\ref{eq1}) is unique. Then
\begin{equation*}
    \|\widetilde{X}-X\|=\|\Delta X\|\leq\theta\|R(\widetilde{X})\|.
\end{equation*}
\end{proof}

%%%%%%%%%%%%%%%%section 6 Numerical Examples%%%%%%%%%%%%%%%%%%%%%%%%%%%%%%%%%%%%%%%%%%%%%%%%%%%%%%%%%
\section{Numerical Examples }
To illustrate the theoretical results of the previous sections, in this
section four simple examples are given, which were carried out
using MATLAB 7.1. For the stopping criterion we take
$\varepsilon_{k+1}(X)=\|X_{k}-A^{*}X_{k}^{-p}A-Q\|<1.0e-10.$

\begin{example}      \label{exam1}                                                     %example(1)
We consider the matrix equation $$X-A^{*}X^{-\frac{1}{3}}A=I,$$
where \[A=\frac{A_{0}}{||A_{0}||},\;\;\;A_{0}=\left(\begin{array}{cc}
2 & 0.95\\
0 & 1
\end{array}\right).
\]
 Suppose that the
coefficient matrix $A$ is perturbed to
$\widetilde{A}=A+\Delta A$, where $$\Delta A=\frac{10^{-j}}{\|C^{T}+C\|}(C^{T}+C)$$ and $C$ is a random matrix generated by
MATLAB function \textbf{randn}.

We compare our own result $\frac{\mu_{*}}{\|X\|}$ in Theorem \ref{thm9}with the perturbation bound $\xi_{*}$
proposed in Theorem 3.1 \cite{s27}.

The condition in Theorem 3.1 \cite{s27} is
\begin{equation*}
con1=\sqrt{\|A\|^2+\zeta}-\|A\|-\|\Delta A\|>0.
\end{equation*}

The conditions in Theorem \ref{thm9} are
\begin{equation*}
con2= 1-\sigma>0,\;\;
  con3 = \frac{l(1-\sigma)^{2}}{\zeta(l+\sigma l+2\eta+2\sqrt{(l\sigma+\eta)(\eta+l)})}-\epsilon >0.
\end{equation*}
By computation, we list them in Table \ref{tab:1}.
%%%%%%%%%%%%%%%%%%%%%%%%%%%%%%%%%%%%%%%%%%%%%%%%%%%%%%%%%%%%%%%%%%%%%%%%%%%%%
\begin{table}[h t b]
\caption{Conditions for Example \ref{exam1} with different values of j}
\label{tab:1}
\begin{tabular}{p{1.5cm}p{2.4cm}p{2.4cm}p{2.4cm}p{2.4cm}}
\hline\noalign{\smallskip}
$j$                & 4             & 5              & 6              & 7\\
\noalign{\smallskip}\hline\noalign{\smallskip}
$con1$ & $0.0455$ & $0.0456$ & $0.0456$ & $0.0456$  \\
$con2$ & $  0.9999$ & $1.0000$ & $1.0000$ & $1.0000$\\
$con3$ & $0.3957$ & $0.3959$ & $ 0.3959$ & $0.3959$\\\hline
\noalign{\smallskip}
\end{tabular}
\end{table}

The results listed in Table \ref{tab:1} show that the conditions in Theorem 3.1 \cite{s27}
and Theorem \ref{thm9} are satisfied.
%%%%%%%%%%%%%%%%%%%%%%%%%%%%%%%%%%%%%%%%%%%%%%%%%%%%%%%%%%%%%%%%%%%%%%%%%%%%%%%%%%%%%%5
\vskip 0.1in

By Theorem 3.1 in \cite{s27} and Theorem \ref{thm9}, we can
compute the relative perturbation bounds $\xi_{*}, \frac{\mu_{*}}{\|X\|}$, respectively. These results averaged as the geometric mean of  10 randomly perturbed runs.
Some results are listed in Table \ref{tab:2}.
%%%%%%%%%%%%%%%%%%%%%%%%%%%%%%%%%%%%%%%%%%%%%%%%%
\begin{table}[h t b]
\caption{Results for Example \ref{exam1} with different values of j}
\label{tab:2}
\begin{tabular}{p{1.5cm}p{2.4cm}p{2.4cm}p{2.4cm}p{2.4cm}}
\hline\noalign{\smallskip}
$j$                & 4             & 5              & 6              & 7\\
\noalign{\smallskip}\hline\noalign{\smallskip}
$\frac{\|\widetilde{X}-X\|}{\|X\|}$ & $6.8119\times 10^{-5}$ & $4.2332\times 10^{-6}$ & $4.3287\times10^{-7}$ & $ 5.5767\times10^{-8}$  \\
 $\xi_{*}$ &  $
2.6003\times 10^{-4}$ & $ 2.1375\times 10^{-5}$ & $ 1.9229\times
10^{-6}$ & $2.7300\times10^{-7}$\\
 $\frac{\mu_{*}}{\|X\|}$ &  $8.8966\times 10^{-5}$ & $ 6.5825\times 10^{-6}$ & $7.2867\times 10^{-7}$ & $9.3455\times 10^{-8}$\\\hline
\noalign{\smallskip}
\end{tabular}
\end{table}

The results listed in Table \ref{tab:2} show that the perturbation
bound $\frac{\mu_{*}}{\|X\|}$ given by Theorem \ref{thm9} is fairly sharp, while the bound $\xi_{*}$ given by Theorem 3.1 in\cite{s27} is conservative.
\end{example}
%%%%%%%%%%%%%%%%%%%%%%%%%%%%%%%%example 2 backward error%%%%%%%%%%%%%%%%%%%%%%%%%%%%%%%%%%%%%%%%%%%%%%%%%%

\begin{example}    \label{exam2}                                                       %example(4)
Consider the equation $$X-A^{*}X^{-3/4}A=Q,$$  for
\[A=\left(\begin{array}{rr}
0.2 & -0.2\\
0.1 & 0.1
\end{array}\right),\;\;\;\;Q=\left(\begin{array}{cc}
0.8939 &0.2987\\
0.1991 &0.6614
\end{array}\right).\] Choose $\widetilde{X}_0=3Q$. Let the
approximate solution $\widetilde{X}_k$ be given with the
iterative method (\ref{eq10}), where $k$ is the iterative number.
Assume that the solution $X$ of Eq.(\ref{eq1}) is
unknown.

We compare our own result with the backward error
proposed in Theorem 4.1 \cite{s27}.

The residual $R(\widetilde{X}_k)\equiv
Q+A^*\widetilde{X}_k^{-p}A-\widetilde{X}_k$ satisfies the conditions in
Theorem 4.1 \cite{s27} and in Theorem \ref{thm5}.

 By Theorem 4.1 in\cite{s27}
, we can compute the backward error
bound
$$\parallel\widetilde{X}_k -X\parallel \leq
\nu_{*}\|R(\widetilde{X}_k)\|, \;\;\mbox{where}\;\;\;
\nu_{*}=\displaystyle\frac{2\|\widetilde{X}_{k}\|\|\widetilde{X}_{k}^{-1}\|}{1-\frac{3}{4}\,\|\widetilde{X}_{k}^{-\frac{3}{8}}A
\widetilde{X}_{k}^{-1/2}\|^{2}}.$$

By Theorem \ref{thm5}, we can compute the new backward error
bound
$$\| \widetilde{X}_{k}-X\|
    \leq\theta\|R(\widetilde{X}_{k})\|,\;\;\mbox{where}\;\;\;\theta=\frac{2\lambda_{\min}(\widetilde{X}_{k})}
    {\theta_{1}+\sqrt{\theta_{1}^{2}-4\lambda_{\min}(\widetilde{X}_{k})\|R(\widetilde{X}_{k})\|}},
    $$$\theta_{1}=(1-\|\widetilde{X}_{k}^{-\frac{3}{8}}A\|^{2}\|\widetilde{X}_{k}^{-1}\|)
    \lambda_{\min}(\widetilde{X}_{k})+\|R(\widetilde{X}_{k})\|.$

    Let $$\kappa_{1}=\frac{\nu_{*}\|R(\widetilde{X}_k)\|}{\parallel\widetilde{X}_k -X\parallel },\;\;
    \;\;\kappa_{2}=\frac{\theta\|R(\widetilde{X}_{k})\|}{\parallel\widetilde{X}_k -X\parallel }.$$

Some results are shown in Table\ref{tab:3}.
\vskip 0.1in
%%%%%%%%%%%%%%%%%%%%%%%%%%%%%%%%%%%%%%%%%%%%%%%%%
\begin{table}[h t b]
\caption{Results for Example \ref{exam2} with different values of k}
\label{tab:3}
\begin{tabular}{p{1.5cm}p{2.4cm}p{2.4cm}p{2.4cm}p{2.4cm}}
\hline\noalign{\smallskip}
$k$                & 4            & 5              & 6              &7\\
\noalign{\smallskip}\hline\noalign{\smallskip}
$||\widetilde{X}_k-X||$ & $6.2131\times 10^{-6}$ & $1.5830\times
10^{-7}$ & $8.2486\times 10^{-9}$ & $6.0132\times 10^{-10}$  \\
 $\nu_{*}|| R(\widetilde{X}_k)||$ & $2.5930\times 10^{-5}$ & $6.6257\times 10^{-7}$ & $3.5697\times 10^{-8}$ & $2.4646\times 10^{-9}$ \\
 $\kappa_{1}$ & $ 4.1734$ & $4.1856$ & $4.3277$ & $ 4.0986$ \\
 $\theta|| R(\widetilde{X}_k)||$ & $ 7.0053\times 10^{-6}$ & $1.7900\times 10^{-7}$ & $9.6440\times 10^{-9}$ & $6.6583\times 10^{-10}$\\
 $\kappa_{2}$ & $ 1.1275$ & $ 1.1308$ & $1.1692$ & $ 1.1073$\\ \hline
\noalign{\smallskip}
\end{tabular}
\end{table}
%%%%%%%%%%%%%%%%%%%%%%%%%%%%%%%%%%%%%%%%%%%%%%%%%%%
\vskip 0.1in
From the results listed in Table \ref{tab:3} we see that the new
backward error bound $\theta|| R(\widetilde{X}_k)||$ is sharper than the backward error bound
$\nu_{*}|| R(\widetilde{X}_k)||$ in \cite{s27}.  Moreover, we see that the
backward error  $\theta ||R(\widetilde{X})||$ for an approximate
solution $\widetilde{X}$ seems to be independent of the
conditioning of the solution $X$.
\end{example}
%%%%%%%%%%%%%%%%%%%%%%%%%%%%%%%%%%%%%%%%%%%%%%%%%%%%%%%%%%%%%%%%%%%%%%%%%%%%%%%%%%%%%%%%%%%%%%%%%%%%%%%%

%%%%%%%%%%%%%%%%%%%%%%%%%%%%%%%%%%%%%%%%%%example(3)%%%%%%%%%%%%%%%%%%%%%%%%%%%%%%%%%%%%%%%%%%%%%%%%%%%%%%%%%%%%%%%
\begin{example}      \label{exam3}                                                     %example(3)
We consider the matrix equation $$X-A^{*}X^{-3}A=5I,$$
where \[A=\frac{A_{0}}{||A_{0}||},\;\;\;A_{0}=\left(\begin{array}{cc}
2 & 0.95\\
0 & 1
\end{array}\right).
\]
We now consider the perturbation bounds for the solution $X$ when the
coefficient matrix $A$ is perturbed to
$\widetilde{A}=A+\Delta A$, where $$\Delta A=\frac{10^{-j}}{\|C^{T}+C\|}(C^{T}+C)$$ and $C$ is a random matrix generated by
MATLAB function \textbf{randn}.

The conditions in Theorem \ref{thm1} are satisfied.

By Theorem \ref{thm1}, we can
compute the relative perturbation bound $\varrho$ with different values of $j$. These results averaged as the geometric mean of  10 randomly perturbed runs.
Some results are listed in Table \ref{tab:4}.
%%%%%%%%%%%%%%%%%%%%%%%%%%%%%%%%%%%%%%%%%%%%%%%%%
\begin{table}[h t b]
\caption{Results for Example \ref{exam3} with different values of j}
\label{tab:4}
\begin{tabular}{p{1.5cm}p{2.4cm}p{2.4cm}p{2.4cm}p{2.4cm}}
\hline\noalign{\smallskip}
$j$                & 4             & 5              & 6              & 7\\
\noalign{\smallskip}\hline\noalign{\smallskip}
$\frac{\|\widetilde{X}-X\|}{\|X\|}$ & $1.1892\times 10^{-7}$ & $2.1101\times 10^{-8}$ & $2.4085\times10^{-9}$ & $ 1.6847\times10^{-10}$  \\
 $\varrho$ &  $
2.0791\times 10^{-7}$ & $ 3.5353\times 10^{-8}$ & $ 3.9573\times
10^{-9}$ & $3.2580\times10^{-10}$\\\hline
\noalign{\smallskip}
\end{tabular}
\end{table}
\vskip 0.1in
The results listed in Table \ref{tab:4} show that the perturbation
bound $\varrho$ given by Theorem \ref{thm1} is fairly sharp.
\end{example}
%%%%%%%%%%%%%%%%%%%%%%%%%%%%%%%%%%%%%%%%%%%%%%%%%%%%%%%%%%%%%%%%%%%%%%%%%%%%%%%%%%
\begin{example}   \label{ex5}                                                           %example(5)
Consider the matrix equation $X-A^{*}X^{-3}A=Q,$ where
\[A=\left(\begin{array}{cc}
0.5& 0.55-10^{-k}\\
1  & 1
\end{array}\right),\;\;\;\;\;\;Q=\left(\begin{array}{cc}
5& 1\\
1  & 5
\end{array}\right).
\]  By Remark \ref{rem2}, we can compute the relative
condition number $c_{rel}(X).$ Some results are listed in Table
\ref{tab:6}.
%%%%%%%%%%%%%%%%%%%%%%%%%%%%%%%%%%%%%%%%%%%%%%%%%
\begin{table}[h t b]
\caption{Results for Example \ref{ex5} with different values of $k$}
\label{tab:6}
\begin{tabular}{p{2cm}p{1.7cm}p{1.7cm}p{1.7cm}p{1.7cm}p{1.7cm}}
\hline\noalign{\smallskip}
  $k$ & 1 & 3 & 5 & 7 & 9 \\ \noalign{\smallskip}\hline\noalign{\smallskip}
 $c_{rel}(X)$& 1.2510 &  1.0991 & 1.0009 & 1.0009 & 1.0009  \\ \hline\noalign{\smallskip}
\end{tabular}
\end{table}

 The numerical results listed in the second line show that the unique
positive definite solution $X$ is well-conditioned.
\end{example}

%%%%%%%%%%%%%%%%section 7 Acknowledgements%%%%%%%%%%%%%%%%%%%%%%%%%%%%%%%%%%%%%%%%%%%%%%%%%%%%%%%%%
\section*{Acknowledgements}
The authors would like to express their gratitude to the referees for their fruitful comments
and suggestions regarding the earlier version of this paper.
%%%%%%%%%%%%%%%%%%%%%%%%%%%%%%%%%%%%%%%%%%%%%%%%%%%%%%%%%%

\end{document}